%% file: enr_factn_sys.tex
\documentclass[12pt]{article}

\input{pre}

\begin{document}

\author{\normalsize  Rory B. B. Lucyshyn-Wright\thanks{The author gratefully acknowledges financial support in the form of an NSERC Postdoctoral Fellowship, as well as earlier, partial financial support in the form of an Ontario Graduate Scholarship.}\let\thefootnote\relax\footnote{Keywords: factorization systems; factorisation systems; enriched categories; strong monomorphisms; strong epimorphisms; monoidal categories; closed categories}\footnote{2010 Mathematics Subject Classification: 18A20, 18A30, 18A32, 18D20}
\\
\small University of Ottawa, 585 King Edward Ave., Ottawa, ON, Canada K1N 6N5}

\title{\large \textbf{Enriched factorization systems}}

\date{}

\maketitle

\abstract{In a paper of 1974, Brian Day employed a notion of factorization system in the context of enriched category theory, replacing the usual diagonal lifting property with a corresponding criterion phrased in terms of hom-objects.  We set forth the basic theory of such enriched factorization systems.  In particular, we establish stability properties for enriched prefactorization systems, we examine the relation of enriched to ordinary factorization systems, and we provide general results for obtaining enriched factorizations by means of wide (co)intersections.  As a special case, we prove results on the existence of enriched factorization systems involving enriched strong monomorphisms or strong epimorphisms.

}

\section{Introduction} \label{sec:intro}

In informal terms, a factorization system on a category $\B$ consists of suitable classes $\E$ and $\M$ of morphisms in $\B$ such that each morphism of $\B$ factors in an essentially unique way as a morphism in $\E$ followed by a morphism in $\M$.  For example the classes of epimorphisms and monomorphisms in the category of sets constitute a factorization system on $\Set$.  In general categories, epi-mono factorizations of a morphism need not be unique (up to isomorphism) even if they exist, but other suitable choices of $\E$ and $\M$ may be available.  Whereas the essential uniqueness of factorizations was mandated directly in the 1948 axiomatics of Mac Lane \cite{MacL:GrCatsDua}, it was realized in the late 1960's (see \cite{Th:FactLocOrthSubc,EhWy} for references) that this uniqueness can be seen as a consequence of a more basic property of mutual complementarity of the classes $\E$ and $\M$:  One stipulates that each $e \in \E$ be \textit{orthogonal} to every $m \in \M$, meaning that for any commutative square as in the periphery of the following diagram
\begin{equation}\label{eqn:orth}
\xymatrix{
A_1 \ar[r] \ar[d]_e & B_1 \ar[d]^m \\
A_2 \ar[r] \ar@{-->}[ur]|w    & B_2
}
\end{equation}
there exists a unique morphism $w$ making the diagram commute.  For example, in many categories $\B$ (e.g., see \bref{rem:str_img_factns_ordinary}, \bref{exa:quasitopos}), orthogonal $(\E,\M)$-factorizations are obtained by setting $\E = \Epi\B$ and taking $\M$ to consist of the \textit{strong} monomorphisms, i.e. those monos to which each epi is orthogonal; dually, one can take $\E$ to consist of the strong epis and $\M$ all monos.  In the example of $\B = \Set$, and indeed in any topos, these canonical choices of factorization system coincide \pbref{exa:quasitopos}.

Realizing that the criterion for orthogonality \eqref{eqn:orth} is equally the statement that the square
\begin{equation}\label{eq:orthog_pb}
\xymatrix {
\B(A_2,B_1) \ar[rr]^{\B(A_2,m)} \ar[d]_{\B(e,B_1)}  & & \B(A_2,B_2) \ar[d]^{\B(e,B_2)} \\
\B(A_1,B_1)  \ar[rr]^{\B(A_1,m)}                     & & \B(A_1,B_2)
}
\end{equation}
be a pullback in $\Set$, Day \cite{Day:AdjFactn} implicitly generalized the notion of factorization system (by then codified in generality in \cite{FrKe}) to the context of \textit{enriched categories}, for which the \textit{hom-objects} $\B(A,B)$ lie in some given monoidal category $\V$, rather than $\Set$.  In this context, one demands that \eqref{eq:orthog_pb} be a pullback in $\V$, thus obtaining a stronger notion of \textit{$\V$-enriched orthogonality}, together with ensuing enriched notions of \textit{prefactorization system} and \textit{factorization system} by analogy with \cite{FrKe}.

Yet to date there has been no substantial published account of the basic theory of such enriched factorization systems and prefactorization systems, notwithstanding their use in a certain special case in \cite{Ke:Ba} \S 6.1, substantial work on related notions in \cite{Ang:thesis,Ang:SemiInitFin}, and a brief treatment of enriched weak factorization systems in \cite{Rie}.  Filling this gap, we establish stability properties for enriched prefactorization systems (\S \bref{sec:stab_canc_enr_prefactn}), we examine the relation of enriched to ordinary factorization systems (\S \bref{sec:charns}), and we provide general results for obtaining enriched factorizations by means of wide (co)intersections (\S \bref{sec:enr_factn_det_cls_monos}).  As a special case, we provide results on the existence of enriched $(\text{Epi},\text{Strong mono})$-factorizations and $(\text{Strong epi},\text{Mono})$-factorizations, where here the notions of epi, mono, strong epi, and strong mono are interpreted in an enriched sense that is in general distinct from the ordinary sense (\S \bref{sec:enr_str_monos}, \bref{sec:enr_factn_det_cls_monos}).  We show that one or both of these canonical enriched factorization systems exists in broad classes of examples, including closed locally presentable or topological categories, categories of algebras of algebraic theories enriched over such \cite{BoDay}, as well as categories of models of weighted-limit sketches enriched over a closed locally presentable category \cite{Ke:Ba}.

The theory of enriched factorization systems was studied and employed in the author's recent Ph.D. thesis \cite{Lu:PhD} in providing a basis for abstract functional analysis in a closed category.

\section{Preliminaries on enriched categories} \label{sec:enr_monos_lims}

In what follows, we work in the context of the theory of categories enriched in a  symmetric monoidal category $\V$, as documented in the seminal paper \cite{EiKe} and the comprehensive references \cite{Ke:Ba}, \cite{Dub}.  We shall include an explicit indication of $\V$ when employing notions such as $\V$-category, $\V$-functor, and so on, omitting the prefix $\V$ only when concerned with the corresponding notions for non-enriched or \textit{ordinary} categories.  When such ordinary notions and terminology are applied to a given $\V$-category $\A$, they should be interpreted relative to the underlying ordinary category of $\A$.  In the absence of any indication to the contrary, we will assume throughout that $\V$ is a \textit{closed} symmetric monoidal category, and in this case we denote by $\uV$ the $\V$-category canonically associated to $\V$, whose underlying ordinary category is isomorphic to $\V$; in particular, the internal homs in $\V$ will therefore be denoted by $\uV(V_1,V_2)$.  We do not assume that any limits or colimits exist in $\V$.

\begin{ParSub}\label{par:cat_classes}
The ordinary categories $\C$ considered in this paper are \textit{not} assumed \textit{locally small}---that is, they are not necessarily $\Set$-enriched categories.  Rather, we assume that for each category $\C$ under consideration, there is a category $\SET$ of \textit{classes} in which lie the hom-classes of $\C$, so that $\C$ is $\SET$-enriched, but $\SET$ is not assumed cartesian closed.
\end{ParSub}

The following notions are defined in \cite{Dub}.

\begin{DefSub} \label{def:enr_mono_limit}
Let $\B$ be a $\V$-category.
\begin{enumerate}
\item A morphism $m:B_1 \rightarrow B_2$ in $\B$ (i.e., in the underlying ordinary category of $\B$) is a \textit{$\V$-mono(morphism)} if $\B(A,m):\B(A,B_1) \rightarrow \B(A,B_2)$ is a monomorphism in $\V$ for every object $A$ of $\B$.  A morphism $e$ in $\B$ is a \textit{$\V$-epi(morphism)} if $e$ is a $\V$-mono in $\B^\op$.
\item $\Mono_\V\B$ and $\Epi_\V\B$ are the classes of all $\V$-monos and $\V$-epis, respectively, in $\B$.
\item A \textit{$\V$-limit} of an ordinary functor $D:\J \rightarrow \B$ consists of a cone for $D$ that is sent by each functor $\B(A,-):\B \rightarrow \V$ ($A \in \B$) to a limit cone for $\B(A,D-)$.  Equivalently, a $\V$-limit is a limit of $D$ that is preserved by each functor $\B(A,-)$.  As special cases of $\V$-limits we define \textit{$\V$-products, $\V$-pullbacks, $\V$-fibre-products}, etc.  \textit{$\V$-colimits} are defined as $\V$-limits in $\B^\op$.
\end{enumerate}
\end{DefSub}
\begin{RemSub}
$\V$-limits coincide with the \textit{conical limits} of \cite{Ke:Ba}.  Note that every $\V$-mono (resp. $\V$-epi, $\V$-limit, $\V$-colimit) in $\B$ is a mono (resp. epi, limit, colimit) in (the underlying ordinary category of) $\B$.
\end{RemSub}

\begin{PropSub} \label{thm:lim_mono_epi_in_tens_cot_vcat_are_enr}
Any ordinary mono (resp. limit) in a tensored $\V$-category $\B$ is a $\V$-mono (resp. $\V$-limit).  Dually, any ordinary epi (resp. colimit) in a cotensored $\V$-category $\B$ is a $\V$-epi (resp. $\V$-colimit).  Hence if $\B$ is tensored, then $\Mono_{\V}\B = \Mono\B$; if $\B$ is cotensored, then $\Epi_{\V}\B = \Epi\B$.
\end{PropSub}
\begin{proof}
If $\B$ is tensored, then each ordinary functor $\B(B,-):\B \rightarrow \V$ is right adjoint (to $(-) \otimes B:\V \rightarrow \B$) and hence preserves limits and monos.
\end{proof}

\begin{PropSub} \label{prop:enr_mono_kp}
A morphism $m:B \rightarrow C$ in a $\V$-category $\B$ is a $\V$-mono if and only if $m$ has a $\V$-kernel-pair $\pi_1,\pi_2:P \rightarrow B$ with $\pi_1 = \pi_2$.
\end{PropSub}
\begin{proof}
If $m$ is a $\V$-mono then $1_B,1_B:B \rightarrow B$ is a $\V$-kernel-pair in $\B$, as one readily checks.  Conversely, if $m$ has a $\V$-kernel-pair with $\pi_1 = \pi_2$, then since each functor $\B(A,-):\B \rightarrow \V$ ($A \in \B$) preserves the given kernel pair, the needed conclusion follows from the analogous result for ordinary categories, which is immediate.
\end{proof}

\begin{CorSub}\label{prop:rvadj_pres_vmonos}
Every right $\V$-adjoint $\V$-functor preserves $\V$-monos.  Dually, every left $\V$-adjoint $\V$-functor preserves $\V$-epis.
\end{CorSub}
\begin{proof}
Right $\V$-adjoints preserve $\V$-limits, so the result follows from \bref{prop:enr_mono_kp}.
\end{proof}

\begin{PropSub} \label{prop:intersection}
Let $(m_i:B_i \rightarrow C)_{i \in \I}$ be a family of $\V$-monos in $\V$-category $\B$, where $\I$ is a class, and let $m:B \rightarrow C$ be a $\V$-fibre-product of this family, with associated projections $\pi_i:B \rightarrow B_i$ ($i \in \I$).  Then $m$ is a $\V$-mono.
\end{PropSub}
\begin{proof}
For each $A \in \B$, we must show that $\B(A,m):\B(A,B) \rightarrow \B(A,C)$ is mono, but $\B(A,m)$ is a fibre product in $\V$ of the monomorphisms $\B(A,m_i)$.  Hence it suffices to show that the analogous proposition holds for ordinary categories, and in this case the verification is straightforward and elementary.
\end{proof}

\begin{DefSub} \label{def:intersection}
In the situation of \ref{prop:intersection}, we say that the $\V$-mono $m$ is a \textit{$\V$-intersection} of the $m_i$.  
\end{DefSub}
\begin{RemSub}
Definition \ref{def:intersection} is an enriched analogue of the notion of intersection defined for ordinary categories in \cite{Ke:MonoEpiPb}.
\end{RemSub}

\begin{DefSub} \label{def:closed_under_tensors}
Given a class of morphisms $\E$ in a $\V$-category $\B$, we say that $\E$ is \textit{closed under tensors} in $\B$ if for any morphism $e:A_1 \rightarrow A_2$ in $\E$ and any object $V \in \V$ for which tensors $V \otimes A_1$ and $V \otimes A_2$ exist in $\B$, the induced morphism $V \otimes e:V \otimes A_1 \rightarrow V \otimes A_2$ lies in $\E$.  Dually, one defines the property of being \textit{closed under cotensors} in $\B$.
\end{DefSub}

\begin{PropSub} \label{thm:closure_props_of_monos}
For any $\V$-category $\B$, the following hold:
\begin{enumerate}
\item If $g \cdot f \in \Mono_\V\B$, then $f \in \Mono_\V\B$.
\item $\Mono_\V\B$ is closed under composition, cotensors, arbitrary $\V$-fibre-products, and $\V$-pullbacks along arbitrary morphisms in $\B$.
\end{enumerate}
\end{PropSub}
\begin{proof}
Both 1 and the needed closure under composition follow immediately from the analogous statements for ordinary categories, upon applying each functor $\B(A,-):\B \rightarrow \V$ ($A \in \B$).  We have already established closure under $\V$-fibre-products in \bref{prop:intersection}, and closure under $\V$-pullbacks is proved by an analogous method, since the corresponding statement for ordinary categories holds.  Lastly, given an object $V$ of $\V$ and a $\V$-mono $m:B_1 \rightarrow B_2$ in $\B$ for which cotensors $[V,B_1], [V,B_2]$ exist in $\B$, the induced morphism $[V,m]:[V,B_1] \rightarrow [V,B_2]$ is $\V$-mono, as follows.  Indeed, for each $A \in \B$, $\B(A,[V,m]) \cong \uV(V,\B(A,m))$ in the arrow category $\Arr\V$ \pbref{exa:pb_cart}, and since $\B(A,m)$ is mono and the right-adjoint functor $\uV(V,-):\V \rightarrow \V$ preserves monos, $\uV(V,\B(A,m))$ is mono.
\end{proof}

\begin{DefSub}
If a morphism $m$ in a $\V$-category $\B$ is a $\V$-equalizer of a given pair of morphisms, then since the equalizer diagram is preserved by each functor $\B(C,-):\B \rightarrow \V$ $(C \in \B)$, it follows that $m$ is a $\V$-mono in $\B$.  We call any such $\V$-equalizer $m$ a \textit{$\V$-regular-mono(morphism)} in $\B$.
\end{DefSub}

\begin{PropSub} \label{prop:section_reg_mono}
Every section in a $\V$-category $\B$ is a $\V$-regular-mono in $\B$.
\end{PropSub}
\begin{proof}
Supposing that the composite $A \xrightarrow{s} B \xrightarrow{r} A$ in $\B$ is the identity morphism $1_A$, it follows that
$$
\xymatrix{
A \ar[r]^s & B \ar[r]^r \ar@/_1.5ex/[rr]_{1_B} & A \ar[r]^s & B
}
$$
is an absolute equalizer diagram and hence a $\V$-equalizer diagram in $\B$.
\end{proof}

\section{Basic notions} \label{sec:enr_factn_sys}

\begin{DefSub} \label{def:enr_orth_fact}
Let $\B$ be a $\V$-category.
\begin{enumerate}
\item For morphisms $e:A_1 \rightarrow A_2$, $m:B_1 \rightarrow B_2$ in $\B$ we say that $e$ is \textit{$\V$-orthogonal} to $m$, written $e \downarrow_\V m$, if the commutative square
\begin{equation}\label{eqn:orth_pb}
\xymatrix {
\B(A_2,B_1) \ar[rr]^{\B(A_2,m)} \ar[d]_{\B(e,B_1)}  & & \B(A_2,B_2) \ar[d]^{\B(e,B_2)} \\
\B(A_1,B_1)  \ar[rr]^{\B(A_1,m)}                     & & \B(A_1,B_2)
}
\end{equation}
is a pullback in $\V$.

\item Given classes $\E$, $\M$ of morphisms in $\B$, we define
$$\E^{\downarrow_\V} := \{m \in \mor\B \;|\; \forall e \in \E \;:\; e \downarrow_\V m\}\;,$$
$$\M^{\uparrow_\V} := \{e \in \mor\B \;|\; \forall m \in \M \;:\; e \downarrow_\V m\}\;.$$

\item A \textit{$\V$-prefactorization-system} on $\B$ is a pair $(\E,\M)$ of classes of morphisms in $\B$ such that $\E^{\downarrow_\V} = \M$ and $\M^{\uparrow_\V} = \E$.

\item For a pair $(\E,\M)$ of classes of morphisms in $\B$, we say that \textit{(\E,\M)-factorizations exist} if every morphism in $\B$ factors as a morphism in $\E$ followed by a morphism in $\M$.  More precisely, we require an assignment to each morphism $f$ of $\B$ an associated pair $(e,m) \in \E \times \M$ with $f = m \cdot e$.

\item A \textit{$\V$-factorization-system} is a $\V$-prefactorization-system such that (\E,\M)-\linebreak[4]factorizations exist.
\end{enumerate}
\end{DefSub}

\begin{RemSub}
See \bref{thm:crit_factn_sys} for an equivalent definition of the notion of $\V$-factorization-system.
\end{RemSub}

\begin{RemSub}
The given definition of orthogonality relative to $\V$ appears in \cite{Day:AdjFactn}, where the notion of enriched factorization system is also implicitly used.
\end{RemSub}

\begin{RemSub}
The above definitions \pbref{def:enr_orth_fact} apply equally when $\V$ is not assumed closed.  In particular, for any ordinary category $\B$, if we take $\V = \SET$ to be a category of classes \pbref{par:cat_classes} in which lie the hom-classes of $\B$, then we recover the analogous notions for ordinary categories as given in \cite{FrKe}, and for these we omit the indication of $\V$ from the notation.  In particular, orthogonality of morphisms in an ordinary category reduces to the diagonal lifting criterion \eqref{eqn:orth}.

For general $\V$, we note that $\V$-orthogonality implies ordinary orthogonality.
\end{RemSub}

\begin{RemSub}
Given a class of morphisms $\sH$ in a $\V$-category $\B$, both $(\sH^{\downarrow_\V\uparrow_\V},\sH^{\downarrow_\V})$ and $(\sH^{\uparrow_\V},\sH^{\uparrow_\V\downarrow_\V})$ are $\V$-prefactorization-systems on $\B$.
\end{RemSub}

\begin{RemSub} \label{rem:factn_dualization}
Each $\V$-(pre)factorization-system $(\E,\M)$ on a $\V$-category $\B$ determines a $\V$-(pre)factorization-system $(\M,\E)$ on $\B^\op$.
\end{RemSub}

Let us record the following well-known and oft-used observations regarding ordinary orthogonality and hence also enriched orthogonality.  The proofs are straightforward.

\begin{PropSub} \label{thm:orth_yielding_retr_and_iso}
Let $g,f$ be morphisms in a category $\B$ such that the composite $g \cdot f$ is defined.
\begin{enumerate}
\item If $g \cdot f \downarrow g$ then $g$ is a retraction in $\B$.
\item If $f \downarrow g \cdot f$ then $f$ is a section in $\B$.
\item If $f \downarrow f$ then $f$ is an isomorphism.
\end{enumerate}
\end{PropSub}

\section{Stability and cancellation for enriched prefactorization systems} \label{sec:stab_canc_enr_prefactn}

In the present section we establish several stability properties of the left and right classes of an enriched prefactorization system.  Many are proved on the basis of analogous properties of the class of \textit{cartesian arrows} of a functor.

\begin{ParSub}\label{par:cart_arr}
Given a functor $P:\A \rightarrow \Y$ and a morphism $f:B \rightarrow C$ in $\A$, recall that $f$ is \textit{$P$-cartesian} if for all morphisms $k:A \rightarrow C$ in $\A$ and $u:PA \rightarrow PB$ in $\Y$ with $Pf \cdot u = Pk$, there exists a unique $v:A \rightarrow B$ in $\A$ with $Pv = u$ and $f \cdot v = k$.  Also, we will say that an object $C \in \A$ is \textit{$P$-monic} if for all objects $D$ of $\A$ and morphisms $x,y:D \rightarrow C$, if $Px = Py$ then $x = y$.
\end{ParSub}

\begin{ExaSub} \label{exa:pb_cart}
Given a category $\Y$, the \textit{arrow category} $\Arr\Y = \CAT(\Two,\Y)$ has as objects all arrows of $\Y$ and as morphisms all commutative squares in $\Y$.  The cartesian morphisms of the codomain functor $\cod:\Arr\Y \rightarrow \Y$  are exactly the pullback squares in $\Y$.  The $\cod$-monic objects are the monomorphisms of $\Y$.
\end{ExaSub}

The following stability properties of cartesian arrows are well-known; the proofs, which are omitted, are easy and elementary exercises.  We also note in passing that cartesian arrows are part of a generalized prefactorization system of the sort considered in \cite{Th:FactConesFunc}\footnote{Indeed, see the first sentence of page 341 there, note that ``$G$-initial'' = ``$G$-cartesian'', and apply (6) of that same paper.}, which in general we might call a \textit{prefactorization system for cones along a functor}, and the given closure properties may alternatively be proved on that basis.

\begin{LemSub} \label{thm:cart_stab}
Let $P:\A \rightarrow \Y$ be a functor, and let $f:B \rightarrow C$, $g:C \rightarrow D$ in $\A$.
\begin{enumerate}
\item If $f$ and $g$ are $P$-cartesian, then $g \cdot f$ is $P$-cartesian.
\item If both $g \cdot f$ and $g$ are $P$-cartesian, then $f$ is $P$-cartesian.
\item If $C$ is $P$-monic and $g \cdot f$ is $P$-cartesian, then $f$ is $P$-cartesian.
\item If $f$ is a pullback of a $P$-cartesian morphism, and this pullback is preserved by $P$, then $f$ is $P$-cartesian.  
\item If $f$ is a fibre product in $\A$ of a family $(f_i:B_i \rightarrow C)_{i \in \I}$ of $P$-cartesian morphisms, indexed by a class $\I$, and this fibre product is preserved by $P$, then $f$ is $P$-cartesian.
\item Suppose $Pf$ is iso.  Then $f$ is $P$-cartesian iff $f$ is iso.
\end{enumerate}
\end{LemSub}

The following stability properties generalize to the enriched context similar properties of ordinary prefactorization systems given in \cite{FrKe} 2.1.1.

\begin{PropSub}\label{thm:pref_stab}
Let $(\E,\M) = (\sH^{\downarrow_\V\uparrow_\V},\sH^{\downarrow_\V})$ be an arbitrary $\V$-prefactorization-system on a $\V$-category $\B$, and let $f:B \rightarrow C$, $g:C \rightarrow D$ be morphisms in $\B$.
\begin{enumerate}
\item If $f, g \in \M$, then $g \cdot f \in \M$.
\item If $g \cdot f \in \M$ and $g \in \M$, then $f \in \M$.
\item Suppose $\sH \subs \Epi_\V\B$.  Then if $g \cdot f \in \M$, it follows that $f \in \M$.
\item If $f$ is a $\V$-pullback in $\B$ of a morphism $m \in \M$, then $f \in \M$.
\item If $f$ is a $\V$-fibre-product in $\B$ of morphisms $f_i:B_i \rightarrow C$ that lie in $\M$ (for $i$ in some class $\I$), then $f \in \M$.
\item Every isomorphism in $\B$ lies in $\M$, and dually, in $\E$.  Hence $\E \cap \M = \Iso\B$ by \bref{thm:orth_yielding_retr_and_iso} \textnormal{3}.
\end{enumerate}
\end{PropSub}
\begin{proof}
Since $\M = \sH^{\downarrow_\V}$ is the intersection of the classes $\{e\}^{\downarrow_\V}$ with $e \in \sH$, we immediately reduce to the case where $\sH$ is a singleton $\{e\}$, so that $\M = \{e\}^{\downarrow_\V}$ for some morphism $e:A_1 \rightarrow A_2$ in $\B$.  There is an ordinary functor 
$$e^*:\B \rightarrow \Arr\V$$
into the category of arrows $\Arr\V$ of $\V$ \pbref{exa:pb_cart}, sending each object $B \in \B$ to the object $e^*(B) := \B(e,B) : \B(A_2,B) \rightarrow \B(A_1,B)$ of $\Arr\V$, and sending each morphism \mbox{$m:B_1 \rightarrow B_2$} in $\B$ to the morphism $e^*(m):e^*(B_1) \rightarrow e^*(B_2)$ defined as the commutative square \eqref{eqn:orth_pb} in $\V$.  Observe that
\begin{enumerate}
\item[(a)] $\M = \{e\}^{\downarrow_\V}$ consists of those morphisms $m$ in $\B$ for which $e^*(m)$ is a cartesian morphism with respect to $\cod:\Arr\V \rightarrow \V$.
\item[(b)] $e^*$ sends $\V$-limits in $\B$ to limits in $\Arr\V$ that are preserved by $\cod:\Arr\V \rightarrow \V$.
\end{enumerate}
Indeed, (a) is immediate from the definitions and \bref{exa:pb_cart}; regarding (b), observe that the composites
$$\B \xrightarrow{e^*} \Arr\V \xrightarrow{\dom} \V,\;\;\;\B \xrightarrow{e^*} \Arr\V \xrightarrow{\cod} \V$$
are $\B(A_2,-)$ and $\B(A_1,-)$, respectively, each of which sends $\V$-limits to limits, so that $e^*$ sends $\V$-limits to pointwise limits in the functor category $\Arr\V = \CAT(\Two,\V)$.

Using (a) and (b), the needed closure properties of $\M$ follow immediately from the corresponding closure properties of cartesian morphisms given in \bref{thm:cart_stab}.  For example, to prove 4, we reason that if $f$ is a $\V$-pullback of $m \in \M$, then by (a) and (b), $e^*(f)$ is a pullback of the $\cod$-cartesian morphism $e^*(m)$ and this pullback is preserved by $\cod$, so that by \bref{thm:cart_stab} 4, $e^*(f)$ is $\cod$-cartesian, i.e. $f \in \M$.  For 3, note that if $e \in \Epi_\V\B$, then $e^*(C) = \B(e,C)$ is mono in $\V$ and hence a $\cod$-monic object (\bref{par:cart_arr},\bref{exa:pb_cart}), so that we can apply \bref{thm:cart_stab} 3.
\end{proof}

When a $\V$-natural family of $\M$-morphisms induces a morphism between $\V$-enriched weighted limits, the resulting morphism again lies in $\M$, as we now show.  The notion of $\V$-enriched weighted limit was called \textit{indexed limit} in \cite{Ke:Ba}.

\begin{PropSub}\label{thm:ind_morph_wlimits_lies_in_m}
Let $(\E,\M)$ be a $\V$-prefactorization-system on a $\V$-category $\B$, let $B,B':\J \rightarrow \B$ be $\V$-functors, and let $m:B \rightarrow B'$ be a $\V$-natural transformation whose component morphisms $m_j:Bj \rightarrow B'j$ $(j \in \J)$ lie in $\M$.  Suppose that \mbox{$W:\J \rightarrow \uV$} is a $\V$-functor for which $\V$-enriched weighted limits $[W,B]$, $[W,B']$ exist in $\B$.  Then the induced morphism $[W,m]:[W,B] \rightarrow [W,B']$ lies in $\M$.
\end{PropSub}
\begin{proof}
Letting $e:A \rightarrow A'$ lie in $\E$, we intend to show that $e \downarrow_\V [W,m]$, i.e. that the diagram
$$
\xymatrix{
\B(A',[W,B]) \ar[d]_{\B(e,[W,B])} \ar[rr]^{\B(A',[W,m])} & & \B(A',[W,B']) \ar[d]^{\B(e,[W,B'])} \\
\B(A,[W,B]) \ar[rr]_{\B(A,[W,m])} & & \B(A,[W,B'])
}
$$
is a pullback in $\V$.  But this diagram is isomorphic to the following diagram
$$
\xymatrix{
[\J,\uV](W,\B(A',B-)) \ar[d]_{[\J,\uV](W,\B(e,B-))} \ar[rrr]^{[\J,\uV](W,\B(A',m_{-}))} & & & [\J,\uV](W,\B(A',B'-)) \ar[d]^{[\J,\uV](W,\B(e,B'-))} \\
[\J,\uV](W,\B(A,B-)) \ar[rrr]_{[\J,\uV](W,\B(A,m_{-}))} & & & [\J,\uV](W,\B(A,B'-)) 
}
$$
(and in particular, the given hom-objects exist in $\V$ and can be taken as just the hom-objects in the preceding diagram).  For each $j \in \J$, we have that $e \downarrow_\V m_j$, so the diagram
$$
\xymatrix{
\B(A',B-) \ar[d]_{\B(e,B-)} \ar[rr]^{\B(A',m_{-})} & & \B(A',B'-) \ar[d]^{\B(e,B'-)} \\
\B(A,B-) \ar[rr]_{\B(A,m_{-})} & & \B(A,B'-) 
}
$$
is a pointwise pullback of $\V$-functors $\J \rightarrow \uV$, and it follows that the preceding diagram is a pullback.
\end{proof}

\begin{CorSub} \label{prop:leftclass_closed_under_tensors}
Let $(\E,\M)$ be a $\V$-prefactorization-system on a $\V$-category $\B$.  Then $\M$ is closed under cotensors in $\B$ \pbref{def:closed_under_tensors}.  Dually, $\E$ is closed under tensors in $\B$.
\end{CorSub}

\section{Characterizations of enriched factorization systems}\label{sec:charns}

\begin{PropSub}\label{thm:enr_pref_as_v_orth_ord_pref}
An ordinary prefactorization system $(\E,\M)$ on a $\V$-category $\B$ is a $\V$-prefactorization-system on $\B$ as soon as each $e \in \E$ is $\V$-orthogonal to each $m \in \M$.   
\end{PropSub}
\begin{proof}
Using our hypothesis of $\V$-orthogonality, we find that
\begin{equation}\E \subs \M^{\uparrow_\V} \subs \M^\uparrow\;,\;\;\;\;\M \subs \E^{\downarrow_\V}\subs \E^\downarrow\;.\end{equation}
But since $(\E,\M)$ is an ordinary prefactorization system, we have that $\E = \M^\uparrow$ and $\M = \E^\downarrow$, so the above inclusions are equalities.
\end{proof}

The following is an enriched analogue of a well-known characterization of ordinary factorization systems that is often used as the definition; see, e.g., \cite{FrKe} 2.2.1, \cite{AHS} 14.6.

\begin{PropSub} \label{thm:crit_factn_sys}
Let $\E$, $\M$ be classes of morphisms in a $\V$-category $\B$.  Then $(\E,\M)$ is a $\V$-factorization-system on $\B$ if and only if the following conditions hold:
\begin{enumerate}
\item Each of $\E$ and $\M$ is closed under composition with isomorphisms.
\item Each $e \in \E$ is $\V$-orthogonal to each $m \in \M$.
\item $(\E,\M)$-factorizations exist.
\end{enumerate}
\end{PropSub}
\begin{proof}
That every $\V$-factorization-system satisfies the given conditions is immediate from \ref{thm:pref_stab} and \ref{def:enr_orth_fact}.  It is well-known that the converse implication holds for ordinary categories (e.g. see \cite{AHS} 14.6).  Hence, in the enriched case, if the given conditions hold, then 2 entails in particular that each $e \in \E$ is orthogonal in the ordinary sense to each $m \in \M$, so the result for ordinary categories entails that $(\E,\M)$ is an ordinary factorization system, and by 2, 3, and \bref{thm:enr_pref_as_v_orth_ord_pref} it then follows that $(\E,\M)$ is a $\V$-factorization-system.
\end{proof}

\begin{CorSub} \label{thm:vfactn_sys_is_factn_sys}
A pair $(\E,\M)$ of classes of morphisms in a $\V$-category $\B$ is a $\V$-factorization-system on $\B$ if and only if $(\E,\M)$ is an ordinary factorization system on $\B$ and each $e \in \E$ is $\V$-orthogonal to each $m \in \M$.
\end{CorSub}
\begin{proof}
That the latter condition entails the former follows from \ref{thm:enr_pref_as_v_orth_ord_pref}.  Conversely, if $(\E,\M)$ is a $\V$-factorization-system, then conditions 1, 2, and 3 of \ref{thm:crit_factn_sys} hold, so the corresponding non-enriched conditions hold with respect to the underlying ordinary category $\B$, and hence the non-enriched analogue of \ref{thm:crit_factn_sys} entails that $(\E,\M)$ is an ordinary factorization system.
\end{proof}

A statement to the effect of the following proposition appears in an entry on the collaborative web site \textit{nLab} \cite{Nlab:enr_factn_sys}; a special case of this observation was also employed earlier in \cite{Ke:Ba} \S 6.1:
\begin{PropSub} \label{thm:enr_orth_from_ord}
Let $\E$ be a class of morphisms in a $\V$-category $\B$.  Suppose that $\B$ is tensored, and suppose that $\E$ is closed under tensors in $\B$ \pbref{def:closed_under_tensors}.  Then $\E^{\downarrow_\V} = \E^{\downarrow}$.

Dually, if a class $\M$ of morphisms in a cotensored $\V$-category $\B$ is closed under cotensors, then $\M^{\uparrow_\V} = \M^\uparrow$.
\end{PropSub}
\begin{proof}
Let $\SET$ be a category of classes \pbref{par:cat_classes} in which lie the hom-classes of $\V$.  Enriched orthogonality implies ordinary, so it suffices to show that $\E^\downarrow \subs \E^{\downarrow_\V}$.  Letting $m:B_1 \rightarrow B_2$ lie in $\E^\downarrow$ and $e:A_1 \rightarrow A_2$ lie in $\E$, we must show that $e \downarrow_\V m$.  It suffices to show that each functor $\V(V,-):\V \rightarrow \SET$ ($V \in \V$) sends the square \eqref{eqn:orth_pb} to a pullback square in $\SET$.  Since $\B$ is tensored, we have in particular that
$$\V(V,\B(A,B)) \cong \B(V \otimes A,B)$$
naturally in $A,B \in \B$.  The diagram in $\SET$ obtained by applying $\V(V,-)$ to the square \eqref{eqn:orth_pb} is therefore isomorphic to the following diagram
$$
\xymatrix {
\B(V \otimes A_2,B_1) \ar[rr]^{\B(V \otimes A_2,m)} \ar[d]_{\B(V \otimes e,B_1)}  && \B(V \otimes A_2,B_2) \ar[d]^{\B(V \otimes e,B_2)} \\
\B(V \otimes A_1,B_1)  \ar[rr]^{\B(V \otimes A_1,m)}                     && {\A(V \otimes A_1,B_2)\;,}
}
$$
which is a pullback in $\SET$ since $V \otimes e \in \E$ and hence $V \otimes e \downarrow m$.
\end{proof}

\begin{PropSub}\label{thm:cotensored_tensored_ord_prefact_is_enr}
Let $(\E,\M)$ be an ordinary prefactorization system on a $\V$-category $\B$.
\begin{enumerate}
\item If $\B$ is tensored and $\E$ is closed under tensors in $\B$, then $(\E,\M)$ is a $\V$-prefactorization-system on $\B$.
\item Dually, if $\B$ is cotensored and $\M$ is closed under cotensors in $\B$, then $(\E,\M)$ is a $\V$-prefactorization-system on $\B$.
\end{enumerate}
\end{PropSub}
\begin{proof}
1. By \bref{thm:enr_pref_as_v_orth_ord_pref} it suffices to show that $\M \subs \E^{\downarrow_\V}$.  But by \ref{thm:enr_orth_from_ord} we have that $\M = \E^\downarrow = \E^{\downarrow_\V}$.
\end{proof}

\begin{PropSub}
Let $\E$, $\M$ be classes of morphisms in a $\V$-category $\B$, and suppose that $\B$ is tensored and cotensored.  Then the following are equivalent:
\begin{enumerate}
\item $(\E,\M)$ is a $\V$-prefactorization-system on $\B$;
\item $(\E,\M)$ is an ordinary prefactorization system on $\B$, $\E$ is closed under tensors, and $\M$ is closed under cotensors.
\end{enumerate}
\end{PropSub}
\begin{proof}
By \ref{thm:cotensored_tensored_ord_prefact_is_enr}, 2 entails 1.  Conversely, if 1 holds, then by \ref{prop:leftclass_closed_under_tensors} we have that $\E$ is closed under tensors and $\M$ is closed under cotensors, so by two applications of \bref{thm:enr_orth_from_ord} we find that $\M = \E^{\downarrow_\V} = \E^\downarrow$ and $\E = \M^{\uparrow_\V} = \M^\uparrow$; therefore $(\E,\M)$ is an ordinary prefactorization system.
\end{proof}

\begin{ThmSub}\label{thm:charn_factn_sys_tensors}
Let $\E$, $\M$ be classes of morphisms in a $\V$-category $\B$.  If $\B$ is tensored, then the following are equivalent:
\begin{enumerate}
\item $(\E,\M)$ is a $\V$-factorization-system on $\B$;
\item $(\E,\M)$ is an ordinary factorization system on $\B$, and $\E$ is closed under tensors in $\B$.
\end{enumerate}
If $\B$ is cotensored, then 1 is equivalent to
\begin{enumerate}
\item[2\textsuperscript{$\prime$}.] $(\E,\M)$ is an ordinary factorization system on $\B$, and $\M$ is closed under cotensors in $\B$.
\end{enumerate}
\end{ThmSub}
\begin{proof}
If $(\E,\M)$ is $\V$-factorization-system, then $(\E,\M)$ is an ordinary factorization system by \ref{thm:vfactn_sys_is_factn_sys}, and $\E$ is closed under tensors by \ref{prop:leftclass_closed_under_tensors}.  Conversely, 2 entails 1 by \ref{thm:cotensored_tensored_ord_prefact_is_enr}.
\end{proof}

\section{Enriched strong monomorphisms and epimorphisms} \label{sec:enr_str_monos}

The notion of strong monomorphism was introduced in \cite{Ke:MonoEpiPb}, and the following enriched generalization of this notion was given in \cite{Day:AdjFactn}.

\begin{DefSub} \label{def:enr_str_mon}
Let $\B$ be a $\V$-category.
\begin{enumerate}
\item A \textit{$\V$-strong-mono(morphism)} in $\B$ is a $\V$-mono $m:B_1 \rightarrow B_2$ such that $e \downarrow_\V m$ for every $\V$-epi $e$ in $\B$.
\item We denote the class of all $\V$-strong-monos in $\B$ by 
$$\StrMono_\V\B := (\Epi_\V\B)^{\downarrow_\V} \cap \Mono_\V\B\;.$$
\item A \textit{$\V$-strong-epi(morphism)} in $\B$ is a $\V$-strong-mono in $\B^\op$, and the class of all such is denoted by $\StrEpi_\V\B$.
\end{enumerate}
\end{DefSub}

\begin{PropSub} \label{thm:stability_and_canc_for_strmonos}
For any $\V$-category $\B$, the following hold:
\begin{enumerate}
\item If $g \cdot f \in \StrMono_\V\B$, then $f \in \StrMono_\V\B$.
\item $\StrMono_\V\B$ is closed under composition, cotensors, arbitrary $\V$-fibre-products, and $\V$-pullbacks along arbitrary morphisms in $\B$.
\end{enumerate}
\end{PropSub}
\begin{proof}
By \bref{sec:stab_canc_enr_prefactn} and \bref{thm:closure_props_of_monos}, each of the classes $(\Epi_\V\B)^{\downarrow_\V}$ and $\Mono_\V\B$ possesses the needed closure properties, so their intersection does as well.
\end{proof}

\begin{PropSub} \label{prop:reg_mono_str_mono}
Every $\V$-regular-mono in a $\V$-category $\B$ is a $\V$-strong-mono.  In particular, every section in $\B$ is a $\V$-strong-mono.
\end{PropSub}
\begin{proof}
Let
$$
\xymatrix{
B_1 \ar[r]^m & B_2 \ar@/^.6ex/[r]^f \ar@/_.6ex/[r]_g & B \\
}
$$
be a $\V$-equalizer diagram in $\B$.  For each $\V$-epi $e:A_1 \rightarrow A_2$ in $\B$ we have a diagram
$$
\xymatrix{
\B(A_2,B_1) \ar[d]_{\B(e,B_1)} \ar[r]^{\B(A_2,m)} & \B(A_2,B_2) \ar[d]_{\B(e,B_2)} \ar@/^.6ex/[r]^{\B(A_2,f)} \ar@/_.6ex/[r]_{\B(A_2,g)} & \B(A_2,B) \ar[d]^{\B(e,B)} \\
\B(A_1,B_1) \ar[r]_{\B(A_1,m)} & \B(A_1,B_2) \ar@/^.6ex/[r]^{\B(A_1,f)} \ar@/_.6ex/[r]_{\B(A_1,g)} & \B(A_1,B)
}
$$
in which each row is an equalizer diagram in $\V$.  Using the fact that each vertical morphism is mono, one readily checks that the leftmost square is a pullback.  The second claim follows from the first via \ref{prop:section_reg_mono}.
\end{proof}

The following is an enriched generalization of part of \cite{FrKe}, 2.1.4.
\begin{PropSub} \label{prop:pref_mono_epi}
Let $(\E,\M)$ be a $\V$-prefactorization-system on a $\V$-category $\B$.
\begin{enumerate}
\item If $\B$ has $\V$-kernel-pairs and every retraction in $\B$ lies in $\E$, then $\M \subs \Mono_\V\B$.
\item Dually, if $\B$ has $\V$-cokernel-pairs and every section in $\B$ lies in $\M$, then $\E \subs \Epi_\V\B$.
\end{enumerate}
\end{PropSub}
\begin{proof}
1.  Let $m:B \rightarrow C$ lie in $\M$.  We have a $\V$-pullback as in the following diagram
$$
\xymatrix{
B \ar@/^/@{=}[drr] \ar@/_/@{=}[ddr] \ar@{-->}[dr]|s &   &  \\
                                      & P \ar[d]|{\pi_1} \ar[r]|{\pi_2} \pullbackcorner & B \ar[d]^m \\
                                      & B \ar[r]_m \ar@{-->}[ur]_w                       & C
}
$$
and an induced $s$ such that $\pi_1 \cdot s = 1_B = \pi_2 \cdot s$.  In particular, $\pi_1$ is a retraction and hence lies in $\E$, so there is a unique $w$ making the diagram commute.  Hence $1_B = \pi_2 \cdot s = w \cdot \pi_1 \cdot s = w \cdot 1_B = w$, so $\pi_2 = w \cdot \pi_1 = 1_B \cdot \pi_1 = \pi_1$ and the result follows by \bref{prop:enr_mono_kp}.
\end{proof}

\begin{PropSub} \label{thm:strmono_rightclass_epi_leftclass}
Let $\B$ be a $\V$-category with $\V$-kernel-pairs.  Then we have the following:
\begin{enumerate}
\item $\StrMono_\V\B = (\Epi_\V\B)^{\downarrow_\V}$.
\item If $\B$ also has $\V$-cokernel-pairs, then $\Epi_\V\B = (\StrMono_\V\B)^{\uparrow_\V}$.
\end{enumerate}
\end{PropSub}
\begin{proof}
Let $(\E,\M) := ((\Epi_\V\B)^{\downarrow_\V\uparrow_\V},(\Epi_\V\B)^{\downarrow_\V})$.  Every retraction in $\B$ is an absolute epi and hence is a $\V$-epi and so lies in $\E$, so by \bref{prop:pref_mono_epi} 1, $\M \subs \Mono_\V\B$ and hence $\StrMono_\V\B = \M \cap \Mono_\V\B = \M$.  If $\B$ also has $\V$-cokernel-pairs, then since every section lies in $\M$ \pbref{prop:reg_mono_str_mono} we deduce by \bref{prop:pref_mono_epi} 2 that $\E \subs \Epi_\V\B$, so $\Epi_\V\B = \E = \M^{\uparrow_\V} = (\StrMono_\V\B)^{\uparrow_\V}$.
\end{proof}

\begin{RemSub} \label{rem:ord_epi_strm_prefactn}
For ordinary categories, one can improve upon \bref{thm:strmono_rightclass_epi_leftclass}.  It was recognized in \cite{CHK} that 2.1.4 of \cite{FrKe} entails that the epimorphisms and strong monomorphisms of a finitely-complete \textit{or} -cocomplete ordinary category together constitute a prefactorization system.  Indeed, by 2.1.4 of \cite{FrKe}, if $\B$ has kernel pairs or finite coproducts, then $\StrMono\B = (\Epi\B)^\downarrow$, and if $\B$ also has cokernel pairs or finite products, then $\Epi\B = (\StrMono\B)^\uparrow$.
\end{RemSub}

On the other hand, if $(\Epi_\V\B,\StrMono_\V\B)$-factorizations are known to exist in $\B$, then the hypotheses of \ref{thm:strmono_rightclass_epi_leftclass} are superfluous:

\begin{PropSub}\label{thm:epi_strmono_factn_sys_if_factns_exists}
If $(\Epi_\V\B,\StrMono_\V\B)$-factorizations exist in a $\V$-category $\B$, then $(\Epi_\V\B,\StrMono_\V\B)$ is a $\V$-factorization system on $\B$.
\end{PropSub}
\begin{proof}
This follows from \bref{thm:crit_factn_sys}.
\end{proof}

\begin{PropSub} \label{thm:ord_str_monos_in_tens_cotens_vcat_are_enriched}
Suppose that $\B$ is a tensored and cotensored $\V$-category.  Then
$\StrMono_{\V}\B = \StrMono\B$, and $\Epi_\V\B = \Epi\B$.
\end{PropSub}
\begin{proof}
The second equation was established in \bref{thm:lim_mono_epi_in_tens_cot_vcat_are_enr}.  For each $V \in \V$, the functor $V \otimes (-):\B \rightarrow \B$ is left adjoint (to $[V,-]:\B \rightarrow \B$) and hence preserves epis.  Hence the class $\E := \Epi\B$ satisfies the hypotheses of \bref{thm:enr_orth_from_ord}, and we compute that
$$(\Epi_\V\B)^{\downarrow_\V} = (\Epi\B)^{\downarrow_\V} = (\Epi\B)^{\downarrow}\;.$$
Therefore, by \bref{thm:lim_mono_epi_in_tens_cot_vcat_are_enr}
$$\StrMono_{\V}\B = (\Epi_\V\B)^{\downarrow_\V} \cap \Mono_\V\B = (\Epi\B)^{\downarrow} \cap \Mono\B = \StrMono\B\;.$$
\end{proof}

\begin{CorSub}\label{thm:tens_cot_epi_strmono_factn_sys_when_factns_exist}
If $\B$ is a tensored and cotensored $\V$-category in which $(\Epi\B,\StrMono\B)$-factorizations exist, then $(\Epi_\V\B,\StrMono_\V\B) = (\Epi\B,\StrMono\B)$ is a $\V$-factorization-system on $\B$.
\end{CorSub}
\begin{proof}
This follows from \bref{thm:epi_strmono_factn_sys_if_factns_exists} and \bref{thm:ord_str_monos_in_tens_cotens_vcat_are_enriched}.
\end{proof}

\begin{ExaSub}\label{exa:quasitopos}
Every quasitopos $\X$ (and in particular, any topos) carries an $\X$-enriched factorization system $(\E,\M)$, where $\E = \Epi\X = \Epi_\X\uX$, $\M = \StrMono\X = \StrMono_\X\uX$, and $\X$ is endowed with its cartesian closed monoidal structure.  Indeed, this follows from \bref{thm:tens_cot_epi_strmono_factn_sys_when_factns_exist} since every morphism in $\X$ factors as an epi followed by a strong mono (\cite{Pen:SurQuTo} 2.10).  Note that the strong monos in $\X$ coincide with the regular monos.  If $\X$ is a topos, then every monomorphism is regular and hence strong, so $\M = \Mono\X = \Mono_\X\uX$; in this case, every epimorphism is regular and hence strong, and so since $\uX^\op$ is tensored and cotensored we deduce by \bref{thm:ord_str_monos_in_tens_cotens_vcat_are_enriched} that $(\E,\M) = (\StrEpi\X,\Mono\X) = (\StrEpi_\X\uX,\Mono_\X\uX)$ as well.
\end{ExaSub}

The following notion is a $\V$-enriched analogue of the notion of \textit{finitely well-complete (ordinary) category} of \cite{CHK}:
\begin{DefSub} \label{def:enr_fwc}
A $\V$-category $\B$ is \textit{$\V$-finitely-well-complete ($\V$-f.w.c.)} if
\begin{enumerate}
\item $\B$ has all finite $\V$-limits, and
\item $\B$ has $\V$-intersections of arbitrary (class-indexed) families of $\V$-strong-monos.
\end{enumerate}
\end{DefSub}
\begin{RemSub} \label{rem:fwc}
By \bref{thm:stability_and_canc_for_strmonos}, the $\V$-intersections of $\V$-strong-monos required in \bref{def:enr_fwc} are necessarily $\V$-strong-monos.
\end{RemSub}

\begin{PropSub} \label{thm:tens_cot_vcat_fwc}
Let $\B$ be a tensored and cotensored $\V$-category.  Then $\B$ is $\V$-finitely-well-complete if and only if $\B$ is finitely well-complete.
\end{PropSub}
\begin{proof}
Finite $\V$-limits, $\V$-intersections, and $\V$-strong-monos in $\B$ are the same as the corresponding ordinary notions (by \bref{thm:lim_mono_epi_in_tens_cot_vcat_are_enr}, \bref{thm:ord_str_monos_in_tens_cotens_vcat_are_enriched}).
\end{proof}

\begin{ExaSub}\label{exa:fwc}
If a small-complete category $\B$ is \textit{well-powered with respect to strong monos}, meaning that each of its objects has but a set of strong subobjects, then $\B$ is f.w.c.  Hence if $\B$ satisfies either of the following conditions, then both $\B$ and $\B^\op$ are f.w.c.:
\begin{enumerate}
\item $\B$ is locally presentable.
\item $\B$ is topological over $\Set$.
\end{enumerate}
Indeed, in case 1, $\B$ is complete, cocomplete, and well-powered (\cite{AdRo:LPrAcCats} 1.56), and $\B$ is also co-well-powered (\cite{AdRo:LPrAcCats} 1.58).  In case 2, the strong monomorphisms in $\B$ are exactly the initial injections, and the strong epimorphisms are the final surjections (e.g. by \cite{Wy:QuTo} 11.9), so that strong subobjects correspond bijectively to subsets of the underlying set, and strong quotients correspond to equivalence relations on the underlying set.

In particular, if $\V$ itself is complete and well-powered with respect to strong monos, then $\uV$ is $\V$-f.w.c. (by \bref{thm:tens_cot_vcat_fwc}, since $\uV$ is tensored and cotensored).  If $\V$ is cocomplete and co-well-powered with respect to strong epis, then $\uV^\op$ is $\V$-f.w.c., again by \bref{thm:tens_cot_vcat_fwc}.  For example, if $\V$ is any symmetric monoidal closed category whose underlying ordinary category is (i) locally presentable  or (ii) topological over $\Set$, then both $\uV$ and $\uV^\op$ are $\V$-finitely-well-complete.
\end{ExaSub}

\begin{PropSub}\label{thm:radj_pres_strmonos}
Every right $\V$-adjoint $\V$-functor preserves $\V$-strong-monos.
\end{PropSub}
\begin{proof}
Let $F \dashv G : \B \rightarrow \A$ be a $\V$-adjunction, and let $m:B_1 \rightarrow B_2$ be a $\V$-strong-mono in $\B$.  By \bref{prop:rvadj_pres_vmonos}, $G$ preserves $\V$-monos, so $Gm$ is $\V$-mono.  Letting $e:A_1 \rightarrow A_2$ be a $\V$-epi in $\A$, we must show that $e \downarrow_\V Gm$, i.e. that the square
$$
\xymatrix{
\A(A_2,GB_1) \ar[d]_{\A(e,GB_1)} \ar[rr]^{\A(A_2,Gm)} & & \A(A_2,GB_2) \ar[d]^{\A(e,GB_2)} \\
\A(A_1,GB_1) \ar[rr]_{\A(A_1,Gm)} & & \A(A_1,GB_2)
}
$$
is a pullback.  But the latter square is isomorphic to the square
$$
\xymatrix{
\B(FA_2,B_1) \ar[d]_{\B(Fe,B_1)} \ar[rr]^{\B(FA_2,m)} & & \B(FA_2,B_2) \ar[d]^{\B(Fe,B_2)} \\
\B(FA_1,B_1) \ar[rr]_{\B(FA_1,m)} & & \B(FA_1,B_2)
}
$$
which is a pullback since $Fe$ is a $\V$-epi (by \bref{prop:rvadj_pres_vmonos}).
\end{proof}

\begin{PropSub}\label{thm:cat_algs_vfwc}\emptybox
\begin{enumerate}
\item Let $F \dashv G : \B \rightarrow \A$ be a $\V$-adjunction with $\A$ $\V$-finitely-well-complete, and suppose that $G$ detects $\V$-limits (i.e., a diagram $D:\J \rightarrow \B$ has a $\V$-limit as soon as its composite with $G$ does).  Then $\B$ is $\V$-finitely-well-complete.
\item Let $\A$ be a $\V$-finitely-well-complete $\V$-category, let $\TT$ be a $\V$-monad on $\A$, and assume that $\V$ has the equalizers needed in order to form the $\V$-category of Eilenberg-Moore algebras $\A^\TT$ \textnormal{(\cite{Dub}, II.1)}.  Then $\A^\TT$ is $\V$-finitely-well-complete.
\end{enumerate}
\end{PropSub}
\begin{proof}
1 follows readily from \bref{thm:radj_pres_strmonos}.  2 follows from 1, since the forgetful $\V$-functor $G:\A^\TT \rightarrow \A$ is a right $\V$-adjoint and detects (indeed \textit{creates}) $\V$-limits.
\end{proof}

\section{Enriched factorization systems by wide (co)intersections} \label{sec:enr_factn_det_cls_monos}

In the present section we obtain several results on the existence of enriched factorization systems when wide intersections of certain kinds of monomorphisms exist.  In particular, we obtain results on the existence of both $(\Epi_\V,\StrMono_\V)$-factorizations as well as $(\StrEpi_\V,\Mono_\V)$-factorizations.  By dualizing, one obtains further results on the existence of such $\V$-factorization-systems when wide \textit{co}intersections exist.  A summary of several of these results is provided in Theorem \bref{thm:summary}.

Let us recall the following result for ordinary categories, which is Lemma 3.1 of \cite{CHK}\footnote{The result and variations thereupon were known earlier, however; e.g. cf. \cite{Th:SemiTop} 6.3, 6.5, 7.3.}.  Whereas the proof is sketched in \cite{CHK}, we include a proof, and then we consider the extent to which this result generalizes to the enriched context, obtaining several corollaries for cotensored enriched categories.

\begin{PropSub}\label{thm:chk_lemma}
Let $\M$ be a class of monomorphisms in an (ordinary) category $\B$.  Suppose that (i) arbitrary intersections of $\M$-morphisms exist in $\B$ and again lie in $\M$, (ii) pullbacks of $\M$-morphisms along arbitrary morphisms exist in $\B$ and again lie in $\M$, and (iii) $\M$ is closed under composition.  Then $(\M^\uparrow,\M)$ is a factorization system on $\B$.
\end{PropSub}
\begin{proof}
Let $g:C \rightarrow B$ in $\B$.  Let $m_0:M \rightarrow B$ be the intersection of all $m:M_m \rightarrow B$ in $\M$ through which $g$ factors --- i.e. for which there exists a (necessarily unique) $g_m:C \rightarrow M_m$ with $m \cdot g_m = g$.  Then $m_0 \in \M$ by hypothesis (i).  The family of all $g_m:C \rightarrow M_m$ induces a unique $e:C \rightarrow M$ such that $m_0 \cdot e = g$ and $\pi_m \cdot e = g_m$ for all $m$, where the morphisms $\pi_m:M \rightarrow M_m$ present $m_0$ as an intersection of the morphisms $m$.

Hence we have a factorization $m_0 \cdot e$ of $g$ with $m_0 \in \M$, and it remains to show that $e \in \M^\uparrow$.  Let $i:B_1 \rightarrow B_2$ lie in $\M$.  For any $h,k$ such that $i \cdot k = h \cdot e$ we obtain a commutative diagram
$$
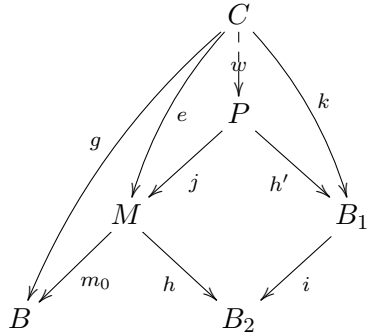
\xymatrix{
  &                       & C \ar@/_2ex/[dddll]_g \ar@/_/[ddl]^{e} \ar@{-->}[d]|w \ar@/^/[ddr]^k & \\
  &                       & P \ar[dl]^{j} \ar[dr]_{h'}                                           & \\
  & M \ar[dl]^{m_0} \ar[dr]_h &                                                                      & B_1 \ar[dl]^i \\
B &                       & B_2                                                                  &  
}
$$
in which the diamond with top vertex $P$ is a pullback and $w$ is the unique morphism into this pullback making the diagram commute.  Since $i \in \M$, the pullback $j$ of $i$ lies in $\M$ by hypothesis (ii).  Hence $m_0 \cdot j \in \M$ since $\M$ is closed under composition, so $m_0 \cdot j$ is a morphism in $\M$ through which $g$ factors, so from the definition of $m_0$ as an intersection it follows that $j$ is iso (since $m_0 \cdot j \cdot \pi_{m_0 \cdot j} = m_0$ and hence the mono $j$ is a retraction of $\pi_{m_0 \cdot j}$ and so is iso).  Letting $u := h' \cdot j^{-1}$, we claim that $u$ is the unique morphism with $u \cdot e = k$ and $i \cdot u = h$.  Indeed, $u \cdot e = h' \cdot j^{-1} \cdot e = h' \cdot w = k$, and $i \cdot u = i \cdot h' \cdot j^{-1} = h \cdot j \cdot j^{-1} = h$, and the uniqueness of $u$ is immediate since $i$ is mono.
\end{proof}

\begin{RemSub}\label{rem:chk_lemma_does_not_generalize}
Notice that the given argument does not directly generalize to the enriched setting, since in showing $e \downarrow i$ we have taken a pullback that depends on the specific commutative square for which a diagonal lift is sought.  However, for \textit{cotensored} enriched categories we obtain the following corollaries.
\end{RemSub}

\begin{CorSub}\label{thm:enr_cor_of_chk}
Let $\M$ be a class of monomorphisms in a cotensored $\V$-category $\B$.  Suppose that the hypotheses of \ref{thm:chk_lemma} hold, and further that $\M$ is closed under cotensors.  Then $(\M^{\uparrow_\V},\M) = (\M^\uparrow,\M)$ is a $\V$-factorization-system on $\B$.
\end{CorSub}
\begin{proof}
By \ref{thm:chk_lemma}, $(\M^\uparrow,\M)$ is a factorization system on $\B$ and hence by \ref{thm:charn_factn_sys_tensors} is a $\V$-factorization-system, and in particular $\M^{\uparrow_\V} = \M^\uparrow$.
\end{proof}

\begin{CorSub} \label{prop:enr_factn_sys_det_rightcls_monos}
Let $(\E,\M)$ be a $\V$-prefactorization-system on a cotensored $\V$-category $\B$, where $\M \subs \Mono_\V\B$.  Suppose $\B$ has arbitrary $\V$-intersections of $\M$-morphisms as well as $\V$-pullbacks of $\M$-morphisms along arbitrary morphisms.  Then the following hold:
\begin{enumerate}
\item $(\E,\M)$ is a $\V$-factorization-system on $\B$.
\item For any class $\Sigma$ of morphisms in $\B$, if we let $\N := \Sigma^{\downarrow_\V} \cap \M$, then $(\N^{\uparrow_\V},\N)$ is a $\V$-factorization-system on $\B$
\end{enumerate}
\end{CorSub}
\begin{proof}
1. $\M$ satisfies the hypotheses of \ref{thm:enr_cor_of_chk}, by \ref{thm:pref_stab}, and the result follows.

2. $\N = (\Sigma^{\downarrow_\V}) \cap (\E^{\downarrow_\V}) = (\Sigma \cup \E)^{\downarrow_\V}$, and setting $\sH := \Sigma \cup \E$ we thus have that $(\N^{\uparrow_\V},\N) = (\sH^{\downarrow_\V\uparrow_\V},\sH^{\downarrow_\V})$ is a $\V$-prefactorization-system on $\B$.  Hence, noting that $\N \subseteq \M \subseteq \Mono_\V\B$, we may apply 1 with respect to $(\N^{\uparrow_\V},\N)$.
\end{proof}

\begin{CorSub}\label{thm:str_mono_vfs_on_vfwc_cot_vcat}
Let $\B$ be a $\V$-finitely-well-complete cotensored $\V$-category.
\begin{enumerate}
\item If $\B$ has $\V$-cokernel-pairs, then $(\Epi_\V\B,\StrMono_\V\B)$ is a $\V$-factorization-system on $\B$
\item For any class $\Sigma$ of morphisms in $\B$, if we let $\N := \Sigma^{\downarrow_\V} \cap \StrMono_\V\B$, then $(\N^{\uparrow_\V},\N)$ is a $\V$-factorization-system on $\B$.
\end{enumerate}
\end{CorSub}
\begin{proof}
Let $\M := \StrMono_\V\B$, $\E := \M^{\uparrow_\V}$.  By \bref{thm:strmono_rightclass_epi_leftclass} 1, $\M = (\Epi_\V\B)^{\downarrow_\V}$, so $(\E,\M) = ((\Epi_\V\B)^{\downarrow_\V\uparrow_\V},(\Epi_\V\B)^{\downarrow_\V})$ is a $\V$-prefactorization-system on $\B$.  We find that $(\E,\M)$ satisfies the hypotheses of \bref{prop:enr_factn_sys_det_rightcls_monos}, and we thus deduce both that $(\E,\M)$ is a $\V$-factorization-system and that 2 holds.  If $\B$ has $\V$-cokernel-pairs, then by \bref{thm:strmono_rightclass_epi_leftclass} 2, $\Epi_\V\B = \M^{\uparrow_\V} = \E$.
\end{proof}

\begin{CorSub} \label{thm:str_img_factns_for_fwc_tens_cot_vcat}
Suppose that $\B$ is a tensored, cotensored, and finitely well-complete $\V$-category.  Then $(\Epi_\V\B,\StrMono_\V\B) = (\Epi\B,\StrMono\B)$ is a $\V$-factorization-system on $\B$.
\end{CorSub}
\begin{proof}
By \ref{thm:ord_str_monos_in_tens_cotens_vcat_are_enriched}, the given equation holds.  By \ref{rem:ord_epi_strm_prefactn}, the given pair is an ordinary prefactorization system, and we find that $\StrMono\B$ satisfies the hypotheses of \ref{thm:chk_lemma}.  Hence $((\StrMono\B)^\uparrow,\StrMono\B) = (\Epi\B,\StrMono\B)$ is a factorization system, and the result now follows by \ref{thm:vfactn_sys_is_factn_sys}.
\end{proof}

\begin{RemSub}\label{rem:str_img_factns_ordinary}
For an ordinary category $\B$, the hypotheses of \bref{thm:str_mono_vfs_on_vfwc_cot_vcat}, \bref{thm:str_img_factns_for_fwc_tens_cot_vcat} can be weakened:  It is shown in \cite{CHK} 3.2 that $(\Epi\B,\StrMono\B)$ is a factorization system as soon as $\B$ is finitely well-complete.  Similarly, the non-enriched analogue of \bref{thm:str_mono_vfs_on_vfwc_cot_vcat} 2 applies to arbitrary finitely well-complete ordinary categories and was employed in a certain instance within the proof of \cite{CHK} 3.3.
\end{RemSub}

\begin{ExaSub}\label{exa:str_factns_in_lp_or_top_v}
By \bref{exa:fwc}, if $\V$ is a symmetric monoidal closed category that is (i) locally presentable or (ii) topological over $\Set$, then both $\uV$ and $\uV^\op$ are finitely well-complete, tensored, and cotensored, so by two invocations of \bref{thm:str_img_factns_for_fwc_tens_cot_vcat} we find that $\uV$ carries $\V$-factorization-systems $(\Epi_\V\uV,\StrMono_\V\uV) = (\Epi\V,\StrMono\V)$ and $(\StrEpi_\V\uV,\Mono_\V\uV) = (\StrEpi\V,\Mono\V)$ .
\end{ExaSub}

\begin{ExaSub}\label{exa:epi_strmono_facts_in_valg_vcat}
Let $\V$ be a finitely well-complete \textit{$\pi$-category} \cite{BoDay}, such as any cartesian closed category that is locally presentable or topological over $\Set$ \pbref{exa:fwc}, and let $\B$ be a \textit{$\V$-algebraic} $\V$-category  \cite{BoDay}.  Since $\uV$ is $\V$-finitely-well-complete \pbref{thm:tens_cot_vcat_fwc} and $\B$ is $\V$-monadic over $\uV$ by \cite{BoDay} 2.2.2, we deduce by \bref{thm:cat_algs_vfwc} that $\B$ is $\V$-finitely-well-complete.  Hence, since $\B$ also has $\V$-cokernel-pairs by \cite{BoDay} 2.3.1 and is cotensored by \cite{Dub} II.4.7, we deduce by \bref{thm:str_mono_vfs_on_vfwc_cot_vcat} that $\B$ carries a $\V$-factorization-system $(\Epi_\V\B,\StrMono_\V\B)$.
\end{ExaSub}

\begin{CorSub}\label{thm:strepi_mono_vfs}
Let $\B$ be a cotensored $\V$-category with $\V$-cokernel-pairs, $\V$-pullbacks of $\V$-monos along arbitrary morphisms, and arbitrary $\V$-intersections of $\V$-monos.  Then $(\StrEpi_\V\B,\Mono_\V\B)$ is a $\V$-factorization-system on $\B$.
\end{CorSub}
\begin{proof}
By \bref{thm:closure_props_of_monos}, $\M := \Mono_\V\B$ satisfies the hypotheses of \bref{thm:enr_cor_of_chk}, so $(\M^{\uparrow_\V},\M)$ is a $\V$-factorization-system.  But since $\B$ has $\V$-cokernel-pairs, the dual of \bref{thm:strmono_rightclass_epi_leftclass} 1 entails that $\M^{\uparrow_\V} = \StrEpi_\V\B$.
\end{proof}

\begin{CorSub}\label{thm:tens_cot_strepi_mono_vfs}
Let $\B$ be a tensored and cotensored $\V$-category with finite limits and arbitrary intersections of monos.  Then $(\StrEpi_\V\B,\Mono_\V\B) = (\StrEpi\B,\Mono\B)$ is a $\V$-factorization-system on $\B$.
\end{CorSub}
\begin{proof}
The analogue of \bref{thm:closure_props_of_monos} for ordinary categories entails that $\Mono\B$ satisfies the hypotheses of \bref{thm:chk_lemma}, so $((\Mono\B)^\uparrow,\Mono\B)$ is a factorization system on $\B$.  But since $\B$ has finite products, the dual of \bref{rem:ord_epi_strm_prefactn} entails that $(\Mono\B)^\uparrow = \StrEpi\B$.  Further, by (the dual of) \bref{thm:ord_str_monos_in_tens_cotens_vcat_are_enriched}, $(\StrEpi_\V\B,\Mono_\V\B) = (\StrEpi\B,\Mono\B)$, and by \bref{thm:vfactn_sys_is_factn_sys}, this factorization system is a $\V$-factorization-system.
\end{proof}

\begin{ExaSub}
Let $\V$ be a well-powered $\pi$-category \cite{BoDay}, such as a locally presentable cartesian closed category.  Again letting $\B$ be a $\V$-algebraic $\V$-category \cite{BoDay} as in \bref{exa:epi_strmono_facts_in_valg_vcat}, $\B$ is cotensored, and $\B$ has all small $\V$-limits and $\V$-colimits by \cite{BoDay} 2.3.1.  Also, since $\V$ is well-powered and $\B$ is monadic over $\V$, it follows that $\B$ is well-powered and hence well-powered with respect to $\V$-monos.  Therefore, by \bref{thm:strepi_mono_vfs}, $(\StrEpi_\V\B,\Mono_\V\B)$ is a $\V$-factorization-system on $\B$.
\end{ExaSub}

\begin{ExaSub}
Suppose that the underlying ordinary category of $\V$ is locally presentable, and let $\B$ be the $\V$-category of models in $\uV$ of a $\V$-enriched \textit{sketch} $(\A^\op,\Phi)$ (\cite{Ke:Ba} \S 6.3).  We deduce via \bref{thm:str_mono_vfs_on_vfwc_cot_vcat} and \bref{thm:strepi_mono_vfs} that $\B$ carries $\V$-factorization-systems \linebreak[4]$(\Epi_\V\B,\StrMono_\V\B)$ and $(\StrEpi_\V\B,\Mono_\V\B)$, as follows.  By \cite{Ke:Ba} Theorem 6.11, $\B$ is a $\V$-reflective sub-$\V$-category of $[\A^\op,\uV]$ and hence is tensored and cotensored and has all small $\V$-limits and $\V$-colimits.  Since $\A^\op$ is a small $\V$-category and $\uV$ is well-powered (with respect to monos and hence $\V$-monos), we deduce by \cite{Dub} IV.1.3 that $[\A^\op,\uV]$ is well-powered with respect to $\V$-monos.  Hence since the inclusion $\B \hookrightarrow [\A^\op,\uV]$ preserves $\V$-monos, $\B$ is well-powered with respect to $\V$-monos and so satisfies the hypotheses of \bref{thm:strepi_mono_vfs} and \bref{thm:str_mono_vfs_on_vfwc_cot_vcat} 1.
\end{ExaSub}

Applying the above results and their duals, we obtain the following.

\begin{ThmSub}\label{thm:summary}
Let $\B$ be a $\V$-category.
\begin{enumerate}
\item $(\Epi_\V\B,\StrMono_\V\B)$ is a $\V$-factorization-system on $\B$ as soon as one of the following holds:
\begin{enumerate}
\item $\B$ is cotensored, well-powered with respect to $\V$-strong-monos, and has small $\V$-limits and $\V$-cokernel-pairs.
\item $\B$ is cotensored and tensored, well-powered with respect to strong monos, and has small limits.
\item $\B$ is tensored, co-well-powered with respect to $\V$-epis, and has small $\V$-colimits and $\V$-kernel-pairs.
\item $\B$ is tensored and cotensored, co-well-powered, and has small colimits.
\end{enumerate}
\item Dually, $(\StrEpi_\V\B,\Mono_\V\B)$ is a $\V$-factorization-system as soon as one of the following holds:
\begin{enumerate}
\item $\B$ is tensored, co-well-powered with respect to $\V$-strong-epis, and has small $\V$-colimits and $\V$-kernel-pairs.
\item $\B$ is tensored and cotensored, co-well-powered with respect to strong epis, and has small colimits.
\item $\B$ is cotensored, well-powered with respect to $\V$-monos, and has small $\V$-limits and $\V$-cokernel pairs.
\item $\B$ is cotensored and tensored, well-powered, and has small limits.
\end{enumerate}
\end{enumerate}
\end{ThmSub}
\begin{proof}
1(a) follows from \bref{thm:str_mono_vfs_on_vfwc_cot_vcat} 1, 1(b) from \bref{thm:str_img_factns_for_fwc_tens_cot_vcat}, 1(c) from the dual of \bref{thm:strepi_mono_vfs}, 1(d) from the dual of \bref{thm:tens_cot_strepi_mono_vfs}.  2 follows from 1 by dualizing.
\end{proof}

\bibliographystyle{amsplain}
\bibliography{bib}

\end{document}

%% file: pre.tex
\usepackage{amsmath, amsthm, amssymb}
\usepackage[all]{xy}
\usepackage{fullpage}
\usepackage{titlesec}
\usepackage{mathrsfs}
\usepackage{amsbsy}
\usepackage{turnstile}
\usepackage{bbm}
\usepackage{yhmath}
\usepackage{tensor}

\usepackage{setspace}

\usepackage[pdftex,bookmarks=true]{hyperref}

\titleformat{\section}{\normalsize\bfseries}{\thesection}{1em}{}
\titleformat{\subsection}{\normalsize\bfseries}{\thesubsection}{1em}{}

\numberwithin{equation}{subsection}

\theoremstyle{plain}

\newtheorem{PropSub}[subsection]{Proposition}
\newtheorem{LemSub}[subsection]{Lemma}
\newtheorem{CorSub}[subsection]{Corollary}
\newtheorem{ThmSub}[subsection]{Theorem}

\theoremstyle{definition}

\newtheorem{DefSub}[subsection]{Definition}
\newtheorem{ExaSub}[subsection]{Example}
\newtheorem{RemSub}[subsection]{Remark}
\newtheorem{ParSub}[subsection]{}

\newcommand*{\emptybox}{\leavevmode\hbox{}}

\newdir{ >}{{}*!/-10pt/@{>}}

\DeclareMathAlphabet{\mathpzc}{OT1}{pzc}{m}{it}
\DeclareMathAlphabet{\mathcalligra}{T1}{calligra}{m}{n}

\newcommand{\bref}[1]{\textnormal{\ref{#1}}}
\newcommand{\pbref}[1]{\textnormal{(\ref{#1})}}

\newcommand{\A}{\ensuremath{\mathscr{A}}}
\newcommand{\B}{\ensuremath{\mathscr{B}}}
\newcommand{\C}{\ensuremath{\mathscr{C}}}

\newcommand{\E}{\ensuremath{\mathscr{E}}}

\newcommand{\sH}{\ensuremath{\mathscr{H}}}
\newcommand{\I}{\ensuremath{\mathscr{I}}}
\newcommand{\J}{\ensuremath{\mathscr{J}}}

\newcommand{\M}{\ensuremath{\mathscr{M}}}
\newcommand{\N}{\ensuremath{\mathscr{N}}}

\newcommand{\V}{\ensuremath{\mathscr{V}}}

\newcommand{\X}{\ensuremath{\mathscr{X}}}
\newcommand{\Y}{\ensuremath{\mathscr{Y}}}

\newcommand{\uX}{\ensuremath{\underline{\mathscr{X}}}}

\newcommand{\uV}{\ensuremath{\underline{\mathscr{V}}}}

\newcommand{\CAT}{\ensuremath{\operatorname{\textnormal{\textbf{CAT}}}}}

\newcommand{\TT}{\ensuremath{\mathbb{T}}}

\newcommand{\mor}{\ensuremath{\operatorname{\textnormal{\textsf{Mor}}}}}

\newcommand{\dom}{\ensuremath{\operatorname{\textnormal{\textsf{dom}}}}}
\newcommand{\cod}{\ensuremath{\operatorname{\textnormal{\textsf{cod}}}}}
\newcommand{\Iso}{\ensuremath{\operatorname{\textnormal{\textsf{Iso}}}}}
\newcommand{\Epi}{\ensuremath{\operatorname{\textnormal{\textsf{Epi}}}}}
\newcommand{\StrEpi}{\ensuremath{\operatorname{\textnormal{\textsf{StrEpi}}}}}
\newcommand{\Mono}{\ensuremath{\operatorname{\textnormal{\textsf{Mono}}}}}
\newcommand{\StrMono}{\ensuremath{\operatorname{\textnormal{\textsf{StrMono}}}}}

\newcommand{\Arr}{\ensuremath{\operatorname{\textnormal{\textsf{Arr}}}}}

\newcommand{\Set}{\ensuremath{\operatorname{\textnormal{\textbf{Set}}}}}
\newcommand{\Two}{\ensuremath{\operatorname{\textnormal{\textbf{2}}}}}
\newcommand{\SET}{\ensuremath{\operatorname{\textnormal{\textbf{SET}}}}}

\newcommand{\subs}{\ensuremath{\subseteq}}

\newcommand{\op}{\ensuremath{\textnormal{op}}}

\newcommand{\pushoutcorner}{\ar@{}[dr]|(.3)\ulcorner}
\newcommand{\pullbackcorner}{\ar@{}[dr]|(.3)\lrcorner}